\documentclass[12pt]{article}
\usepackage{pdfsync}
\usepackage{setspace,mathtools,amsmath,amssymb,amsthm,fullpage,outline,centernot} 
\usepackage[pdftex]{graphicx}
\usepackage{mathrsfs}
\usepackage{enumerate}
\usepackage{accents}
\theoremstyle{plain}
\newtheorem{theorem}{Theorem}
\newtheorem{lemma}[theorem]{Lemma}
\newtheorem{corollary}[theorem]{Corollary}
\newtheorem{proposition}[theorem]{Proposition}

\theoremstyle{definition}

\newtheorem{conjecture}[theorem]{Conjecture}

\theoremstyle{remark}

\newcommand{\Z}{\mathbb{Z}}

\newcommand{\half}{\frac{1}{2}}
\newcommand{\wo}{\backslash}

\def\bal#1\eal{\begin{align*}#1\end{align*}}
\newcommand{\inv}{^{-1}}
\DeclarePairedDelimiter{\floor}{\lfloor}{\rfloor}
\DeclarePairedDelimiter{\ceil}{\lceil}{\rceil}

\usepackage{hyperref}

\begin{document}
\title{Unimodality of Partitions in Near-Rectangular Ferrers Diagrams}
\author{Samuel Zbarsky\\
\small Carnegie Mellon University\\
\small\tt sa\_zbarsky@yahoo.com\\
}
\date{}
\maketitle

\section{Abstract}
We look at the rank generating function $G_\lambda$ of partitions inside the Ferrers diagram of some partition $\lambda$, investigated by Stanton in 1990, as well as a closely related problem investigated by Stanley and Zanello in 2013. We show that $G_\lambda$ is not unimodal for a larger class of 4-part partitions than previously known, and also that if the ratios of parts of $\lambda$ are close enough to 1 (depending on how many parts $\lambda$ has), or if the first part is at least half the size of $\lambda$, then $G_\lambda$ is unimodal.

\section{Introduction}
Given partition $\lambda$ with $\lambda=(\lambda_1,\ldots,\lambda_b)$, the \emph{length} of $\lambda$ is $b$ and the \emph{size} of $\lambda$, denoted $|\lambda|$, is $\sum \lambda_i$. Given partitions $\mu$ and $\lambda$, we say that $\mu\subseteq\lambda$ if the Ferrers diagram of $\mu$ fits inside the Ferrers diagram of $\lambda$. For any partition $\lambda$, we look at the set of partitions $\mu\subseteq\lambda$, ordered by the relation above, and let $G_\lambda$ be the rank-generating function of this poset (i.e.\ $G_\lambda(q)=\sum a_nq^n$ where $a_n$ is the number of partitions $\mu\subseteq\lambda$ of size $n$). We are interested in understanding when the coefficients of $G_\lambda$ are unimodal, i.e.\ first weakly increasing and then weakly decreasing. This question was considered by Stanton~\cite{Stanton1990}.

Let $b$ be the length of $\lambda$. In the case that $\lambda=(a,\ldots,a)$, we get that $G_\lambda(q)$ is the Gaussian binomial coefficient $G_\lambda(q)=\binom{a+b}{b}_q$. There are many proofs that these are unimodal symmetric (for instance, \cite{Sylvester1973,Proctor1982,OHara1990,Zeilberger1989}). This fact will be used many times in the proof of our main result.

We can also look at the question above, but require that both $\lambda$ and $\mu$ have distinct nonzero parts. We call this rank generating function $F_\lambda(q)$. The unimodality of $F_\lambda(q)$ was considered by Stanley and Zanello~\cite{StanleyZanello2013}. For this version, the partition $\lambda=(b,b-1,\ldots,1)$ gives the generating function
\[
F_\lambda(q)=\prod_{i=1}^b (1+q^b)
\]
which is unimodal symmetric (see, for instance, \cite{Proctor1982}). Alpoge~\cite{Alpoge2013} proved that for $n\gg b$, if $\lambda=(n,n-1,\ldots,n-b+1)$, then $F_\lambda$ is unimodal. We will mainly be concerned with Stanton's problem, but some of our results will also extend to the Stanley-Zanello version. For both problems, it is proven by Stanton and Stanley-Zanello respectively that any partition with at most 3 parts has a unimodal generating function, and that there are infinite families of partitions with 4 parts that do not.

In section~\ref{sec: genfunc}, we will derive a form for $F_\lambda$ and $G_\lambda$ when the length of $\lambda$ is fixed and show that $F_\lambda$ and $G_\lambda$ are unimodal if $2\lambda_1\ge |\lambda|$. In section~\ref{sec: 4parts}, we will prove that for a reasonable notion of positive density, for both versions of the problem, there is a positive density of partitions with 4 parts that have nonunimodal generating functions, giving larger classes of nonunimodal partitions than were known before. In sections~\ref{sec: ascenddescend} and~\ref{sec: concave}, we will prove that partitions with 5 or more parts that are near-rectangular (in a sense to be defined later in the paper) and with $|\lambda|\gg b$ have unimodal $G_\lambda$. Finally, in section~\ref{sec: conjectures}, we present some conjectures that are supported by the results in this paper and by computational data.

\section{Generating Function}\label{sec: genfunc}
We will use $[n]$ to denote the set $\{1,2,\ldots,n\}$. We begin by finding an expression for $G_\lambda(q)$ when $\lambda$ has a fixed number of parts $b$.

If $\lambda$ is a partition, denote by $G^i_\lambda$ the size generating function for the number of partitions $\mu\subseteq\lambda$, giving the first part of the partition the weight $i$ (for instance, the partition $(4,2,1)$ will contribute to the $q^{15}$ coefficient of $G^3_\lambda$ since 3(4)+2+1=15). Note that $G^1_\lambda=G_\lambda$. Define $F^i_\lambda$ similarly.

Also, denote by $\bar{\lambda}$ the partition $\lambda$ with the first part removed. We consider generating functions of the form $G^i_{\lambda}$. Let $\lambda$ have $b$ parts, for $b$ fixed. For any $A\subseteq [b]$ and any $1\le k\le b$, we define $f^\lambda_A(k)$ and $g_A(k)$ as follows: if $k\ge \min(A)$, define $f^\lambda_A(k)=\lambda_{b+1-\max(A\cap [k])}+1$. Otherwise, define $f^\lambda_A(k)=0$.

Define $g_A(k)=k-\max(A\cap [k])+1$. For convenience, here and in the formula below, we define $\max(\emptyset)=1$.
\begin{proposition}
\begin{equation}\label{Gigenfunc}
G^i_{\lambda}(q)=\sum_{A\subseteq [b]}G^i_{\lambda, A}(q)
\end{equation}
where
\[
G^i_{\lambda, A}(q)=(-1)^{|A|}q^{\displaystyle\left(\sum_{k=1}^{b-1} f^\lambda_A(k)\right)+if^\lambda_A(b)}\prod_{k=1}^{\max(A)-1} \frac{1}{1-q^{g_A(k)}}\cdot\prod_{k=\max(A)}^{b}\frac{1}{1-q^{g_A(k)+i-1}}.
\]
\end{proposition}
The idea for the inductive  step in the proof below is based on Lemma~1 from \cite{Stanton1990} (where it is only used once).
\begin{proof}
We induct on $b$. For $b=1$,
\[
G^i_{\lambda}(q)=1+q^i+\ldots+q^{i\lambda_1}=\frac{1-q^{i(\lambda_1+1)}}{1-q^i}=\frac{1}{1-q^i}-\frac{q^{i(\lambda_1+1)}}{1-q^i},
\]
which matches \eqref{Gigenfunc}.

Assume $b>1$. If $\mu\subseteq\lambda$ with $\mu_1<\lambda_1$, then we let $\phi(\mu)=(\mu_1+1,\mu_2,\ldots,\mu_b)$. If $\nu\subseteq\lambda$ with $\nu_1>\nu_2$, we let $\phi\inv(\nu)=(\nu_1-1,\nu_2,\ldots,\nu_b)$. For any $n$, clearly $\phi$ gives a bijection between partitions $\mu\subseteq\lambda$ with $|\mu|=n-1$ and $\mu_1<\lambda_1$ and partitions $\nu\subseteq\lambda$ with $|\nu|=n$ and $\nu_1=\nu_2$. Thus taking $(1-q^i)G^i_{\lambda}(q)$ gives a lot of cancellation. Specifically, we get
\begin{align}
(1-q^i)G^i_{\lambda}(q)&=G^{i+1}_{\bar{\lambda}}(q)-q^i(q^{i\lambda_1})G^1_{\bar{\lambda}}(q) \label{(1-q)G}\\
       G^i_{\lambda}(q)&=\frac{1}{(1-q^i)}G^{i+1}_{\bar{\lambda}}(q)-\frac{1}{(1-q^i)}(q^{i(\lambda_1+1)})G^1_{\bar{\lambda}}(q) \label{recursion}
\end{align}
where the first term on the right side of \eqref{(1-q)G} corresponds to $\nu$ with $\nu_1=\nu_2$, while the second term corresponds to $\mu$ with $\mu_1=\lambda_1$.

It is not difficult (with care) to verify that for $A\subseteq [b-1]$,
\[
G^i_{\lambda, A}(q)=\frac{1}{(1-q^i)}G^{i+1}_{\bar{\lambda}, A}(q)
\]
and
\[
G^i_{\lambda, A\cup\{b\}}(q)=-\frac{q^{i(\lambda_1+1)}}{(1-q^i)}G^1_{\bar{\lambda}, A}(q),
\]
so by $\eqref{recursion}$, we get
\bal
G^i_{\lambda}(q)&=\frac{1}{(1-q^i)}G^{i+1}_{\bar{\lambda}}(q)-\frac{1}{(1-q^i)}(q^{i(\lambda_1+1)})G^1_{\bar{\lambda}}(q)\\
                &=\sum_{A\subseteq [b-1]}\frac{1}{(1-q^i)}G^{i+1}_{\bar{\lambda}, A}(q)-\sum_{A\subseteq [b-1]}\frac{q^{i(\lambda_1+1)}}{(1-q^i)}G^{1}_{\bar{\lambda}, A}(q)\\
							  &=\sum_{B\subseteq [b]} G^i_{\lambda, B}(q).
\eal
\end{proof}

We can substitute $i=1$ (which is the case we are really interested in).
\begin{corollary}
\begin{equation}\label{Ggenfunc}
G_{\lambda}(q)=\sum_{A\subseteq [b]}\left((-1)^{|A|}q^{\sum_{k=1}^{b} f^\lambda_A(k)}\prod_{k=1}^{b} \frac{1}{1-q^{g_A(k)}}\right)
\end{equation}
\end{corollary}
Note that for a fixed $b$, this gives us $G_{\lambda}(q)$ as a sum of $2^b$ terms which are simple to compute given $\lambda$.

To get the generating function for distinct parts, we note that
\begin{equation}\label{Fgenfunc}
F_{\lambda}(q)=1+\sum_{c=1}^bq^{\binom{c+1}{2}}G_{(\lambda_1-c,\ldots,\lambda_c-1)}(q)
\end{equation}
where $c$ above represents the number of nonzero parts in a partition.

Observe that $G_{\lambda}(q)$ is unimodal if and only if the coefficients of $(1-q)G_{\lambda}(q)$ go from being nonnegative to being nonpositive.
We can also use $\eqref{(1-q)G}$ with $i=1$ to obtain directly the sign of coefficients when $|\bar{\lambda}|\le\lambda_1$.
\begin{corollary}\label{Gfirstlarge}
If $2\lambda_1\ge|\lambda|$, then $G_\lambda$ is unimodal
\end{corollary}
Similarly to $\eqref{(1-q)G}$, we can derive
\[
(1-q)F_{\lambda}(q)=qF^2_{\bar{\lambda}}(q)+1-q^{\lambda_1+1}F_{\bar{\lambda}}(q)
\]
which gives us
\begin{corollary}\label{Ffirstlarge}
If $2\lambda_1\ge|\lambda|$, then $F_\lambda$ is unimodal
\end{corollary}

\section{Partitions of length 4}\label{sec: 4parts}
Some set $S$ of partitions with $b$ parts has \emph{positive density} if
\[
\liminf_{n\to\infty}\frac{|\{\lambda\in S\mid |\lambda|\le n\}|}{|\{\text{partitions $\lambda$ with $b$ parts and $|\lambda|\le n$}\}|}>0.
\]
Note that corrolaries~\ref{Gfirstlarge} and \ref{Ffirstlarge} imply that a positive density of partitions with 4 parts have unimodal $F_\lambda$ and $G_\lambda$.
\begin{theorem}
A positive density of partitions $\lambda$ with 4 parts have nonunimodal $G_\lambda$ and a positive density of partitions $\mu$ with 4 distinct parts have nonunimodal $F_\mu$.
\end{theorem}
\begin{proof}
The generating function gives us the following when $b=4$:
\bal
G_{\lambda}(q)=&\frac{1-q^{4\lambda_4+4}}{(1-q)(1-q^2)(1-q^3)(1-q^4)}+\frac{q^{3\lambda_3+\lambda_4+4}+q^{\lambda_1+3\lambda_4+4}-q^{3\lambda_3+3}-q^{\lambda_1+1}}{(1-q)^2(1-q^2)(1-q^3)}\\
&+\frac{q^{2\lambda_2+\lambda_3+3}+q^{\lambda_1+2\lambda_3+3}+q^{\lambda_1+\lambda_2+2}-q^{2\lambda_2+\lambda_3+\lambda_4+4}-q^{\lambda_1+2\lambda_3+\lambda_4+4}+q^{\lambda_1+\lambda_2+2\lambda_4+4}}{(1-q)^3(1-q^2)}\\
&+\frac{q^{\lambda_1+\lambda_2+\lambda_3+\lambda_4+4}-q^{\lambda_1+\lambda_2+\lambda_3+3}}{(1-q)^4}+\frac{q^{2\lambda_2+2\lambda_4+4}-q^{2\lambda_2+2}}{(1-q)^2(1-q^2)^2}.
\eal

We will look at cases when $\lambda_4>0.9\lambda_1$ and when $\lambda_1\equiv\lambda_2\equiv 11\pmod{12}$. For $\lambda_1+\lambda_2+14\le N<2\lambda_1+2$ and $12\mid N$, we will be interested in cases where $[q^{N+1}](1-q)G_{\lambda}(q)<0$, and $[q^{N+2}](1-q)G_{\lambda}(q)>0$, since such cases show that $G_{\lambda}(q)$ is not unimodal. We will also assume that $\lambda_1$ is sufficiently large.

Define $m,n,\ell$ so that $\lambda_1=12m-1$, $\lambda_2=12n-1$, and $N=12(\ell-1)$.
Since $3\lambda_4>N+2$, the only relevant terms in $(1-q)G_{\lambda}(q)$ are
\[
\frac{1}{(1-q^2)(1-q^3)(1-q^4)}-\frac{q^{12m}}{(1-q)(1-q^2)(1-q^3)}+\frac{q^{12(m+n)}}{(1-q)^2(1-q^2)}-\frac{q^{24n}}{(1-q)(1-q^2)^2}.
\]
We turn all the denominators into the form $(1-q)^{12}$ and look at coefficients of $q^i$ in the numerators for $i$ being 1 or 2 mod 12. This gives us
\bal
[q^{N+1}](1-q)G_{\lambda}(q)&=9\ell^2-(48m+15)\ell+24m^2+72mn-36n^2+38m+6\\
[q^{N+2}](1-q)G_{\lambda}(q)&=9\ell^2-(48m+15)\ell+24m^2+72mn-36n^2+34m+6n+6.\\
\eal
We will now restrict further to the case that $1.999m<\ell<1.9999m$ and $0.98m<n<\ell-m$. Let
\bal
f(m,\ell,n)=[q^{N+1}](1-q)G_{\lambda}(q)&=9\ell^2-(48m+15)\ell+24m^2+72mn-36n^2+38m+6\\
g(m,\ell,n)=[q^{N+2}](1-q)G_{\lambda}(q)&=9\ell^2-(48m+15)\ell+24m^2+72mn-36n^2+34m+6n+6.\\
\eal

First, we note that $g(m,\ell,n-1)-f(m,\ell,n)=-78(m-n)+2m-42>0$. We also note that 
\[
g(\ell,m,\ell-m-1)=-27\ell^2+96\ell m-84m^2\pm O(m). 
\]
Using $1.999m<\ell<1.9999m$ and neglecting linear terms since $\lambda_1$ is sufficiently large gives us $g(\ell,m,\ell-m-1)>0$. Similarly,
\[
g(\ell,m,\ceil{0.98m}+1)=9\ell^2-48m\ell+59.9856m^2\pm O(m)<0. 
\]
Thus for any $m,\ell$, there must exist some $n$ within the bounds defined above such that $g(\ell,m,n)>0$ and $g(\ell,m,n-1)\ge 0$. Then $f(\ell,m,n)< g(m,\ell,n-1)\le 0$. Since $[q^N](1-q)G_{\lambda}(q)>[q^{N+2}](1-q)G_{\lambda}(q)$, we then have
\bal
[q^{N+1}](1-q)G_{\lambda}(q)&<0\\
[q^{N+2}](1-q)G_{\lambda}(q)&>0,\\
\eal
so $G_\lambda(q)$ is nonunimodal at $N+1$.

Thus there are $\Theta(a^2)$ triples $(n,m,\ell)$ with $m\le a$ that satisfy all the conditions above and are not unimodal, so there are $\Theta(a^4)$ pairs $(\lambda,N)$ with $\lambda_1\le a$ that satisfy all the conditions above and are not unimodal. Finally, we can check that $g(m,\ell+1,n)<f(m,\ell,n)$ and $g(m,\ell,n)$ is decreasing in $\ell$ over the appropriate range, so for any pair of values $(m,n)$, there is at most one $\ell$ with $f(\ell,m,n)<0$ and $g(\ell,m,n)>0$. Thus we are not double-counting any partitions. Hence there are $\Theta(a^4)$ partitions $\lambda$ with 4 parts $\lambda_1\le a$, and nonunimodal $G_\lambda$, so a positive density of 4-part partitions $\lambda$ have nonunimodal $G_\lambda$, proving the first part of the theorem.

If $\mu$ has 4 parts, then
\[
F_\mu(q)=q^6G_{(\mu_1-2,\mu_2-1,\mu_3)}(q)+qG_{(\mu_1-1,\mu_2)}(q)+1
\]
where the first term covers partitions $\nu\subseteq\mu$ with 3 or 4 parts, the second term covers $\nu$ with 1 or 2 parts, and the third term covers the empty partition. Whenever $\lambda$ is a partition with nonunimodal $G_\lambda$ of the form described above, taking $\mu=(\lambda_1+3,\lambda_2+2,\lambda_3+1,\lambda_4)$ gives us
\bal
[q^{N+7}](1-q)F_{\mu}(q)=[q^{N+1}](1-q)G_{\lambda}(q)&<0\\
[q^{N+8}](1-q)F_{\mu}(q)=[q^{N+2}](1-q)G_{\lambda}(q)&>0\\
\eal
since $\deg(qG_{(\mu_1-1,\mu_2)}(q)+1)=\mu_1-1+\mu_2+1=\lambda_1+\lambda_2+5<N+6$. Thus $F_\mu$ is not unimodal. This proves the theorem.
\end{proof}

\section{Ascending and Descending Segments}\label{sec: ascenddescend}
We will try to prove that, under certain conditions on $\lambda$, we get long initial increasing and final decreasing sequences of coefficients of $G_\lambda$. To prove the increasing sequences, we will use the unimodality of the largest rectangle that fits inside $\lambda$, while for the decreasing sequences we will use the unimodality of the smallest rectangle that contains $\lambda$.

\begin{proposition} \label{equalincreasingweak}
For any partition $\lambda$ and any $k\ge 1$, the sequence $a_0,\ldots,a_{\ceil{k\lambda_k/2}}$ of coefficients of $G_\lambda$ is weakly increasing.
\end{proposition}
\begin{proof}
Let $X$ be the rectangle of squares $(x,y)$ in the Ferrers diagram of $\lambda$ satisfying $x\le k$ and $y\le\lambda_k$. Let $R\subseteq\lambda\wo X$ be some fixed set of squares. Look at partitions $\mu$ such that $\mu\subseteq\lambda$ and $\mu\wo X=R$ and assume that there is at least one such $\mu$. Let $b_n$ be the number of such $\mu$ of size $n$. Some number (possibly 0) of rows and columns of $X$ are forced to be in $\mu$, and the remaining squares of $X$ form a rectangle. Thus $\{b_i\}$ is unimodal symmetric, centered at an index of at least $k\lambda_k/2$, so $\{b_i\}$ is weakly increasing up to at least $b_{\ceil{k\lambda_k/2}}$. Since $\{a_i\}$ is the sum of such sequences over all possible choices of $R$, we get that $\{a_i\}$ is weakly increasing up to at least $a_{\ceil{k\lambda_k/2}}$.
\end{proof}

\begin{proposition}\label{equaldecreasingstrong} 
For any partition $\lambda$ of length $b$ with $\lambda_{b-1}\ge\frac{2b-5}{2b-4}\lambda_1$ and $\lambda_b\ge\frac{|\lambda|}{b+1}$, if we let $N=\floor{b\lambda_1/2}$, then the sequence $a_N,a_{N+1},\ldots$ is weakly decreasing.
\end{proposition}
\begin{proof}
Let $m$ be the number of indices $1\le k<b$ such that $\lambda_k>\lambda_{k+1}$. We prove by induction on $m$ that $a_N,a_{N+1},\ldots$ is weakly decreasing.

If $m=0$, then $\lambda$ is a rectangle and thus $G_\lambda$ is unimodal symmetric, so $a_N,a_{N+1},\ldots$ is weakly decreasing. This proves the base case.

Now assume $m\ge 1$. Take maximum $k$ with $k<b$ such that $\lambda_k>\lambda_{k+1}$. Then $\lambda=(\lambda_1,\ldots,\lambda_k,\lambda_{k+1},\ldots,\lambda_{k+1})$.
Let $\nu=(\lambda_1,\ldots,\lambda_k,\lambda_k,\ldots,\lambda_k)$. Then $G_\nu$ weakly decreases from $q^N$ onward by the inductive hypothesis. Let $H=G_\nu-G_\lambda$. Then $H$ counts partitions $\mu$ inside $\nu$ with $\mu_{k+1}>\lambda_{k+1}$. For any $S=(s_1,\ldots,s_{k+1})$ such that $s_i\le\nu_i$ for all $i$ and $\lambda_{k+1}<s_{k+1}\le s_k\le\cdots\le s_1$, let $H^S$ be the generating function for partitions $\mu$ inside $\nu$ with $\mu_i=s_i$ for all $1\le i\le k+1$. Each ordered tuple of values $(\mu_{k+2},\ldots,\mu_b)$ such that $s_{k+1}\ge\mu_{k+2}\ge\cdots\ge\mu_b\ge 0$ (i.e.\ each partition fitting within a certain rectangle) gives us a partition $\mu=(s_1,\ldots,s_{k+1},\mu_{k+2},\ldots,\mu_b)$ (when $k+1=b$, there is exactly one such partition).Thus $H^S$ is unimodal symmetric, centered at index  $\sum_{i=1}^{k+1} s_i+(b-k-1)s_{k+1}/2$.

Hence $H^S$ is weakly increasing up to $q^L$ for $L=\sum_{i=1}^{k+1}s_i+\ceil{s_{k+1}(b-k-1)/2}$. Since
\bal
\sum_{i=1}^{k+1}s_i+\ceil{s_{k+1}(b-k-1)/2}&\ge (k+1)s_{k+1}+\ceil{s_{k+1}(b-k-1)/2}\\
																					 &\ge \ceil{s_{k+1}(b+k+1)/2}\\
																					 &\ge \ceil{(\lambda_{k+1}+1)(b+k+1)/2},
\eal
we know that $H^S$ is weakly increasing up to $q^M$ for $M=\ceil{(\lambda_{k+1}+1)(b+k+1)/2}$. Since $H=\sum_S H^S$, we know that $H$ is weakly increasing up to $q^M$. Since $G_\lambda=G_\nu-H$, we know that $G_\lambda$ is weakly decreasing from $q^N$ to $q^M$.

We now assume that $k<b-1$ ($k=b-1$ will be a separate case) and let $X$ be the set of squares $(x,y)$ in the Ferrers diagram of $\lambda$ such that $x\ge k+1$ or $y>\lambda_{x+1}$:

\includegraphics[scale=0.2]{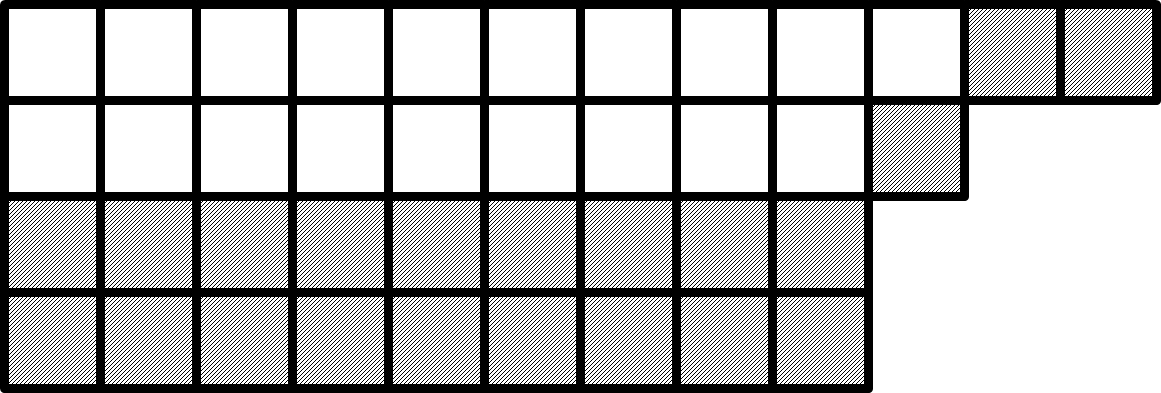}

This gives us $|X|=\lambda_{k+1}\left(b-k-1\right)+\lambda_1$. Note that $X$ is a disjoint union of a set of rectangles such that for any two rectangles in the set, one lies strictly above and to the right of the other. 

Let $R\subseteq\lambda\wo X$ be some fixed set of squares in the Ferrers diagram of $\lambda$. Look at $\mu\subseteq\lambda$ such that $\mu\wo X=R$ and assume that there is at least one such $\mu$. Let $b_n$ be the number of such $\mu$ of size $n$. Some squares in $X$ are forced to be not in $\mu$, and the remaining squares form a disjoint union of a set of rectangles such that for any two rectangles in the set, one lies strictly above and to the right of the other. The generating function for partitions inside each rectangle is unimodal symmetric, and the product of unimodal symmetric generating functions is unimodal symmetric, so $\{b_i\}$ is unimodal symmetric, centered at an index which is at most $|\lambda|-\half|X|$. Since $\{a_i\}$ is the sum of such sequences over all possible choices of $R$, we know that $\{a_i\}$ is weakly decreasing from $a_{|\lambda|-\ceil*{\half|X|}}$ onward.

Since 
\[
\lambda_{k+1}\ge\lambda_{b-1}\ge\frac{2b-5}{2b-4}\lambda_1\ge\frac{2k-1}{2k}\lambda_1,
\]
we get
\bal
M+\ceil*{\half|X|}&=\ceil*{\half(\lambda_{k+1}+1)(b+k+1)}+\ceil*{\half\left(\lambda_{k+1}\left(b-k-1\right)+\lambda_1\right)}\\
                 &\ge \half\lambda_{k+1}(b+k+1)+\half\lambda_{k+1}(b-k-1)+\half\lambda_1\\
								 &\ge (b-k)\lambda_{k+1}+k\lambda_{k+1}+\half\lambda_1\\
								 &\ge (b-k)\lambda_{k+1}+\left(k\left(\frac{2k-1}{2k}\right)+\half\right)\lambda_1\\
								 &\ge (b-k)\lambda_{k+1}+k\lambda_1\\
								 &\ge |\lambda|.
\eal
Then, since we know that $\{a_i\}$ is weakly decreasing from $a_{|\lambda|-\ceil*{\half|X|}}$ onward and also weakly decreasing from $a_N$ to $a_M$, we have that $\{a_i\}$ is weakly decreasing from $a_N$ onward, so we are done with the inductive step.

We look separately at the case $k=b-1$. There, for each $n>|\lambda|-\lambda_b$, we have an injection from partitions $\mu\subseteq\lambda$ of size $n$ to partitions $\nu\subseteq\lambda$ of size $n-1$ given by $\phi(\mu)=(\mu_1,\ldots,\mu_{b-1},\mu_b-1)$. Thus $\{a_i\}$ is weakly decreasing from $a_{|\lambda|-\lambda_b}$ onward.
We get
\bal
M+\lambda_b&=\ceil*{\half(\lambda_b+1)(2b)}+\lambda_b\\
           &\ge b\lambda_b+\lambda_b\\
					 &\ge (b+1)\lambda_b\\
					 &\ge |\lambda|.
\eal
Then, since we know that $\{a_i\}$ is weakly decreasing from $a_{|\lambda|-\lambda_b}$ onward and also weakly decreasing from $a_N$ to $a_M$, we have that $\{a_i\}$ is weakly decreasing from $a_N$ onward, so we are done with the inductive step.

\end{proof}

\section{Concavity}\label{sec: concave}
\begin{lemma}\label{quasipolyslemma}
If $C$ is a finite multiset of positive integers with $\gcd(C)=1$, then
\[
[x^n]\prod_{c\in C} \frac{1}{1-q^c}=\prod_{c\in C}\frac{1}{c}\binom{n}{|C|-1}\pm O_C(n^{|C|-2}).
\]
\end{lemma}
\begin{proof}
This is easy to see by induction on $|C|$.
\end{proof}

\begin{proposition} \label{5partsequalconcave}
When $b=5$, $\lambda_4\ge\frac{5}{6}\lambda_1$, $\lambda_5\ge\half\lambda_1$, and $\lambda_1$ is sufficiently large, the coefficients of $G_\lambda(q)$ are concave from $q^{\frac{5}{2}\lambda_2}$ to $q^{\frac{5}{2}\lambda_1}$.
\end{proposition}

\begin{proof}
Concavity is equivalent to the negativity of the coefficients of $(1-q)^2G_\lambda(q)$. We will look at the generating function for $G_\lambda(q)$ derived above, however most of its terms will not affect the coefficients we are interested in. For $\frac{5}{2}\lambda_2\le n\le\frac{5}{2}\lambda_1$, we get

\bal
[q^n](1-q)^2G_\lambda(q)&=[q^n]\left(\frac{1-q}{(1-q^2)(1-q^3)(1-q^4)(1-q^5)}-\frac{q^{2\lambda_2+2}}{(1-q^2)^2(1-q^3)}\right.\\
                        &\left.-\frac{q^{\lambda_1+1}}{(1-q^2)(1-q^3)(1-q^4)}+\frac{q^{\lambda_1+\lambda_2+2}}{(1-q)(1-q^2)(1-q^3)}\right)\\
												&=[q^n]\left(\frac{1-q+q^2}{(1-q^4)(1-q^5)(1-q^6)}-\frac{q^{2\lambda_2+2}}{(1-q^2)^2(1-q^3)}\right.\\
												&\left.-\frac{q^{\lambda_1+1}}{(1-q^2)(1-q^3)(1-q^4)}+\frac{q^{\lambda_1+\lambda_2+2}}{(1-q)(1-q^2)(1-q^3)}\right)\\
												&=\frac{1}{120}\binom{n}{2}-\frac{1}{12}\binom{n-2\lambda_2}{2}-\frac{1}{24}\binom{n-\lambda_1}{2}+\frac{1}{6}\binom{n-\lambda_1-\lambda_2}{2}\pm O(n)\\
												&=\frac{n^2}{240}-\frac{\left(n-2\lambda_2\right)^2}{24}-\frac{\left(n-\lambda_1\right)^2}{48}+\frac{\left(n-\lambda_1-\lambda_2\right)^2}{12}\pm O(n)\\
												&=\left(\frac{\left(\frac{n}{\lambda_1}\right)^2}{240}-\frac{\left(\frac{n}{\lambda_1}-2\frac{\lambda_2}{\lambda_1}\right)^2}{24}-\frac{\left(\frac{n}{\lambda_1}-1\right)^2}{48}+\frac{\left(\frac{n}{\lambda_1}-\frac{\lambda_2}{\lambda_1}-1\right)^2}{12}\right)\lambda_1^2\pm O(n).\\
\eal
If we define
\[
\alpha(x,y)=\frac{y^2}{240}-\frac{\left(y-2x\right)^2}{24}-\frac{\left(y-1\right)^2}{48}+\frac{\left(y-x-1\right)^2}{12},
\]
then we care about $\alpha(x,y)$ when $\frac{5}{6}\le x\le 1$ and $\frac{5}{2}x\le y\le \frac{5}{2}$. Note that $\alpha(x,y)$ is convex in $y$, and we can check that $\alpha(x,\frac{5}{2})<-0.002$ and $\alpha(x,\frac{5}{2}x)<-0.002$. Thus $\alpha(\frac{\lambda_2}{\lambda_1},\frac{n}{\lambda_1})<-0.002$. Then if $\lambda_1\gg 1$, the quadratic term dominates (since $n$ is linear in $\lambda_1$), so $[q^n](1-q)^2G_\lambda(q)$ is negative, so the coefficients of $G_\lambda(q)$ are concave from $q^{\frac{5}{2}\lambda_2}$ to $q^{\frac{5}{2}\lambda_1}$.
\end{proof}

\begin{theorem}\label{5nearrectunimodal}
If partition $\lambda$ has 5 parts, with $\lambda_4\ge \frac{5}{6}\lambda_1$, $\lambda_5\ge\frac{|\lambda|}{6}$, and $\lambda_1$ being sufficiently large, then $G_\lambda$ is unimodal. 
\end{theorem}
\begin{proof}
The condition on $\lambda_5$ gives us
\bal
6\lambda_5&\ge\lambda_1+\lambda_2+\lambda_3+\lambda_4+\lambda_5\\
5\lambda_5&\ge\lambda_1+\frac{5}{6}\lambda_1+\frac{5}{6}\lambda_1+\frac{5}{6}\lambda_1\\
 \lambda_5&\ge\frac{7}{10}\lambda_1.
\eal
Let $\mu=(\lambda_1,\lambda_2,\lambda_2,\lambda_2,\lambda_2)$. By Proposition~\ref{equalincreasingweak}, $G_\mu$ is increasing up to $q^{\frac{5}{2}\lambda_2}$. Since $G_\mu$ and $G_\lambda$ agree up to at least $q^{\min(5\lambda_5,4\lambda_4,3\lambda_3)}$ 
we get that $G_\lambda$ has increasing coefficients up to $q^{\frac{5}{2}\lambda_2}$.

Then by Proposition~\ref{5partsequalconcave}, the coefficients of $G_\lambda$ are concave from $q^{\frac{5}{2}\lambda_2}$ to $q^{\frac{5}{2}\lambda_1}$, and by Proposition~\ref{equaldecreasingstrong}, they are decreasing from $q^{\frac{5}{2}\lambda_1}$ onward. Thus $G_\lambda$ is unimodal.
\end{proof}

Before proving a similar result for all $b\ge 5$, we establish a lemma.
\begin{lemma}\label{nonzeroaltsum}
Let $T(b)=\sum_{i=0}^{\floor{b/2}}(-1)^i(b-2i)^{b-3}\binom{b}{i}$. Then $T(b)\ne 0$ for $b\ge 5$.
\end{lemma}
\begin{proof}
For $b<100$, this is easy to check directly using, say, Mathematica.

Nagura~\cite{Nagura1952} proved that whenever $x\ge 25$, there is a prime number $p$ with $x<p<\frac{6}{5}x$. When $b\ge 100$, we can set $x=\frac{b}{4}$ to obtain some odd prime $p$ with $3p<b<4p$. Pick $k$ so that $b=3p+k$. Then $1\le k<p$. We want to show that $T(b)\not\equiv 0\bmod{p}$. Note that $\binom{b}{i}\not\equiv 0\bmod{p}$ only when $0\le i\le k$ or $0\le i-p\le k$. Thus
\bal
T(3p+k)&\equiv \sum_{i=0}^k (-1)^i(3p+k-2i)^{3p+k-3}\binom{3p+k}{i}+\sum_{j=0}^k (-1)^{j+p}(3p+k-2(j+p))^{3p+k-3}\binom{3p+k}{p+j}\\
       &\equiv \sum_{i=0}^k (-1)^i(k-2i)^{3p+k-3}\binom{3p+k}{i}-\sum_{j=0}^k (-1)^j(k-2j)^{3p+k-3}\binom{3p+k}{p+j}\\
		   &\equiv \sum_{i=0}^k (-1)^i(k-2i)^{3p+k-3}\binom{k}{i}-3\sum_{j=0}^k (-1)^j(k-2j)^{3p+k-3}\binom{k}{j}\\
		   &\equiv -2\sum_{j=0}^k (-1)^j(k-2j)^{3(p-1)+k}\binom{k}{j}\\
		   &\equiv -2\sum_{j=0}^k (-1)^j(k-2j)^k\binom{k}{j}.\\
\eal
Note that $\sum_{j=0}^k (-1)^j(k-2j)^k\binom{k}{j}$ is, up to a possible sign, the expression for the $k$-th partial difference of $(k-2x)^k$, which is $(-1)^k2^kk!$ and thus is not divisible by $p$. Therefore $T(3p+k)\not\equiv 0\bmod p$, so $T(3p+k)\ne 0$.
\end{proof}

\begin{theorem}\label{bnearrectunimodal}
For a given integer $b\ge 5$, there exists $\epsilon>0$ such that if $\lambda$ is a partition with $b$ parts satisfying $\lambda_1\gg b$, $\lambda_{\floor{b/2}}>(1-\epsilon)\lambda_1$, and $\lambda_b\ge\frac{2b-5}{2b-4}\lambda_1$, then $G_\lambda$ is unimodal.
\end{theorem}
\begin{proof}
Fix $b$. If we let $\mu=(\lambda_1,\ldots,\lambda_{\floor{b/2}},\lambda_{\floor{b/2}},\ldots,\lambda_{\floor{b/2}})$, then we know by Proposition~$\ref{equaldecreasingstrong}$ that $G_\mu$ has increasing coefficients up to $q^{\frac{b}{2}\lambda_{\floor{b/2}}}$. Since the coefficients of $G_\mu$ and $G_\lambda$ agree up to at least $q^{(\floor{b/2}+1)\lambda_b}$ and 
\[
\left(\floor*{\frac{b}{2}}+1\right)\lambda_b\ge\frac{b+1}{2}\frac{2b-5}{2b-4}\lambda_1=\frac{b}{2}\frac{b+1}{b}\frac{2b-5}{2b-4}\lambda_1=\frac{b}{2}\left(\frac{2b^2-3b-5}{2b^2-4b}\right)\lambda_1\ge\frac{b}{2}\lambda_1\ge\frac{b}{2}\lambda_{\floor{b/2}},
\]
we get that $G_\lambda$ has increasing coefficients up to $q^{\frac{b}{2}\lambda_{\floor{b/2}}}$.
Also, Proposition~$\ref{equaldecreasingstrong}$ tells us that $G_\lambda$ has decreasing coefficients from $q^{\frac{b}{2}\lambda_1}$ onward. Thus all we need to do is prove concavity of coefficients between $q^{\frac{b}{2}\lambda_{\floor{b/2}}}$ and $q^{\frac{b}{2}\lambda_b}$. For this, we look at $[q^n](1-q)^2G_\lambda(q)$ for $\frac{b}{2}\lambda_{\floor{b/2}}\le n\le \frac{b}{2}\lambda_b$.
Note that the only term in \eqref{Ggenfunc} that does not have $(1-q)^2$ in the denominator is $\prod_{k=1}^{b} \frac{1}{1-q^k}$. For that term, we use $\frac{(1-q)^2}{(1-q)(1-q^2)(1-q^3)(1-q^4)(1-q^5)}=\frac{1-q+q^2}{(1-q^6)(1-q^4)(1-q^5)}$. Thus $(1-q)^2G_\lambda(q)$ will have only terms of the form covered by Lemma~\ref{quasipolyslemma}.

Using $h(A,\lambda)$ to denote $\sum_{k=1}^{b} f^\lambda_A(k)$, we obtain
\bal
[q^n](1-q)^2G_{\lambda}(q)&=[q^n](1-q)^2\sum_{A\subseteq [b]}\left((-1)^{|A|}q^{h(A,\lambda)}\prod_{k=1}^{b} \frac{1}{1-q^{g_A(k)}}\right)\\
													&=[q^n](1-q)^2\sum_{\substack{A\subseteq [b]\\ h(A,\lambda)\le n}} \left((-1)^{|A|}q^{h(A,\lambda)}\prod_{k=1}^{b} \frac{1}{1-q^{g_A(k)}}\right)\\
													&=\sum_{\substack{A\subseteq [b]\\ h(A,\lambda)\le n}} \left((-1)^{|A|}\left(\prod_{k=1}^b\frac{1}{g_A(k)}\right)\binom{n-h(A,\lambda)}{b-3}+O_b(n^{b-4})\right)\\
													&=\alpha\left(\frac{\lambda_2}{\lambda_1},\ldots,\frac{\lambda_{\floor{b/2}}}{\lambda_1},\frac{n}{\lambda_1}\right)n^{b-3}+O_b(n^{b-4})\\
\eal
where $\alpha$ is some continuous function. Note that in going from line 2 to line 3 in the display, we have to treat the case of $A=\emptyset$ separately.

Now, whenever $\alpha\left(1,\ldots,1,\frac{b}{2}\right)<0$, we can use continuity to obtain some $\epsilon>0$ so that for any $\lambda$ and $n$ with $\lambda_{\floor{b/2}}\ge(1-\epsilon)\lambda_1$ and $\frac{b}{2}\lambda_{\floor{b/2}}\le n\le\frac{b}{2}\lambda_1$, we get $\alpha(\frac{\lambda_2}{\lambda_1},\ldots,\frac{\lambda_{\floor{b/2}}}{\lambda_1},\frac{n}{\lambda_1})<\alpha(1,\ldots,1,\frac{b}{2})/2<0$. This would then tell us, when $\lambda_1\gg b$, that the coefficients of $G_{\lambda}(q)$ are concave from $q^{(b/2)\lambda_{\floor{b/2}}}$ to $q^{(b/2)\lambda_1}$. Thus all we need to do is show that $\alpha\left(1,\ldots,1,\frac{b}{2}\right)<0$.

We let $\lambda=(\lambda_1,\ldots,\lambda_1)$. Then $h(A,\lambda)=(b-\min(A)+1)(\lambda_1+1)$ for $A\ne\emptyset$ and $h(\emptyset,\lambda)=0$. Thus

\bal
[q^n]&(1-q)^2G_{\lambda}(q)=\\
                          &=\sum_{\substack{A\subseteq [b]\\ h(A,\lambda)\le \frac{b}{2}\lambda_1}} \left((-1)^{|A|}\left(\prod_{k=1}^b\frac{1}{g_A(k)}\right)\binom{n-h(A,\lambda)}{b-3}+O_b(n^{b-4})\right)\\
                          &=\left(\prod_{k=1}^b\frac{1}{g_\emptyset(k)}\right)\binom{n}{b-3}+\sum_{m=\ceil{\frac{b+3}{2}}}^b\sum_{\substack{A\subseteq [b]\\ \min(A)=m}}\left((-1)^{|A|}\left(\prod_{k=1}^b\frac{1}{g_A(k)}\right)\binom{n-h(A,\lambda)}{b-3}\right)+O_b(n^{b-4})\\
													&=\frac{1}{b!}\binom{\frac{b}{2}\lambda_1}{b-3}+\sum_{m=\ceil{\frac{b+3}{2}}}^b\sum_{\substack{s_1,\ldots,s_\ell\ge 1\\\sum s_j=b-m+1}}\left(\frac{1}{(b-1)!}\prod_{j=0}^\ell\frac{1}{s_j!}\right)\binom{\frac{b}{2}\lambda_1-(b-m+1)(\lambda_1+1)}{b-3}+O_b(\lambda_1^{b-4})\\
													&=\frac{1}{(b-3)!}\Bigg(\frac{1}{b!}\left(\frac{b}{2}\lambda_1\right)^{b-3}+\\
													&\text{\hspace{30 pt}}+\sum_{m=\ceil{\frac{b+3}{2}}}^b\sum_{\substack{s_1,\ldots,s_\ell\ge 1\\\sum s_j=b-m+1}}\left(\frac{1}{(m-1)!}\prod_{j=0}^\ell\frac{-1}{s_j!}\right)\left(\left(\frac{b}{2}-(b-m+1)\right)\lambda_1\right)^{b-3}\Bigg)+O_b(\lambda_1^{b-4})\\
													&=\frac{\lambda_1^{b-3}}{(b-3)!2^{b-3}}\left(\frac{1}{b!}b^{b-3}+\sum_{m=\ceil{\frac{b+3}{2}}}^b\sum_{\substack{s_1,\ldots,s_\ell\ge 1\\\sum s_j=b-m+1}}\left(\frac{1}{(m-1)!}\prod_{j=0}^\ell\frac{-1}{s_j!}\right)\left(2(m-1)-b\right)^{b-3}\right)+O_b(\lambda_1^{b-4}).\\
\eal
where $\displaystyle\sum_{\substack{s_1,\ldots,s_\ell\ge 1\\\sum s_j=x}}$ denotes a sum over compositions of $x$.

We note the equality of formal power series 
\bal
\sum_{a=0}^\infty \left(\sum_{\substack{s_1,\ldots,s_\ell\ge 1\\\sum s_j=a}}\prod_{j=0}^\ell\frac{-1}{s_j!}\right)x^a&=\frac{1}{1-\left(\sum_{s=1}^\infty\frac{-1}{s!}x^s\right)}\\
                                                                               &=\frac{1}{1-(1-e^x)}\\
                                                                               &=e^{-x}\\
																																							 &=\sum_{a=0}^\infty \frac{(-1)^a}{a!}x^a,\\
\eal
so
\begin{equation}\label{simplifytosanify}
\sum_{\substack{s_1,\ldots,s_\ell\ge 1\\\sum s_j=a}}\prod_{j=0}^k\frac{-1}{s_j!}=\frac{(-1)^a}{a!}.
\end{equation}
Substituting \eqref{simplifytosanify} in, we get
\bal
[q^n]&(1-q)^2G_{\lambda}(q)=\\
                          &=\frac{\lambda_1^{b-3}}{(b-3)!2^{b-3}}\left(\frac{1}{b!}b^{b-3}+\sum_{m=\ceil{\frac{b+3}{2}}}^b\left(\frac{1}{(m-1)!}\frac{(-1)^{b-m+1}}{(b-m+1)!}\right)\left(2(m-1)-b\right)^{b-3}\right)+O_b(\lambda_1^{b-4})\\
                          &=\frac{\lambda_1^{b-3}}{b!(b-3)!2^{b-3}}\left(b^{b-3}+\sum_{m=\ceil{\frac{b+3}{2}}}^b\left(\binom{b}{b-m+1}(-1)^{b-m+1}\right)\left(2(m-1)-b\right)^{b-3}\right)+O_b(\lambda_1^{b-4})\\
													&=\frac{\lambda_1^{b-3}}{b!(b-3)!2^{b-3}}\sum_{i=0}^{\floor{(b-1)/2}}(-1)^i(b-2i)^{b-3}\binom{b}{i}+O_b(\lambda_1^{b-4})\\
													&=\frac{\lambda_1^{b-3}}{b!(b-3)!2^{b-3}}\sum_{i=0}^{\floor{b/2}}(-1)^i(b-2i)^{b-3}\binom{b}{i}+O_b(\lambda_1^{b-4})
\eal
so
\[
\alpha\left(1,\ldots,1,\frac{b}{2}\right)=\frac{\lambda_1^{b-3}}{b!(b-3)!2^{b-3}}\sum_{i=0}^{\floor{b/2}}(-1)^i(b-2i)^{b-3}\binom{b}{i}.
\]
This value is not 0 by Lemma~\ref{nonzeroaltsum}. Note that if $\alpha(1,\ldots,1,\frac{b}{2})>0$, then when $\lambda=(\lambda_1,\ldots,\lambda_1)$ and $\lambda_1\gg b$, we get that the coefficients of $G_\lambda$ are strictly convex near the middle, which contradicts its unimodality since $(\lambda_1,\ldots,\lambda_1)$ is a rectangular partition. Thus $\alpha(1,\ldots,1,\frac{b}{2})<0$ and the theorem is proved.
\end{proof}

\section{Conjectures}\label{sec: conjectures}
We used the generating functions \eqref{Ggenfunc} and \eqref{Fgenfunc} to test unimodality of partitions with a small number of parts, leading us to pose the following conjecture.
\begin{conjecture}\label{mostunimodal}
If $\lambda_1\ge b>4$\footnote{An earlier version of this paper incorrectly omitted the condition $\lambda_1\ge b$. Thank you to Henry Cohn for pointing this out. Note that $\lambda_1\ge b$ always holds for $F_\lambda$ and holds without loss of generality for $G_\lambda$ since taking the transpose of $\lambda$ does not change the generating function.}, then $F_\lambda$ and $G_\lambda$ are unimodal except for finitely many exceptions when $b=6$.
\end{conjecture}
 The maximal exception for $G_\lambda$ appears to be $(10,9,9,9,9,9)$, while the maximal exception for $F_\lambda$ appears to be $(19,18,17,16,15,14)$. We have used a computer to test this for all partitions with 5 parts with $\lambda_1\le 200$, with 6 parts with $\lambda_1\le 100$, with 7 parts with $\lambda_1\le 70$, with 8 parts with $\lambda_1\le 50$, with 9 parts with $\lambda_1\le 40$, and with 10 parts with $\lambda_1\le 30$.

Conjecture~\ref{mostunimodal} would also tend to be supported by Theorem~\ref{bnearrectunimodal}. However, it seems very hard to prove.

Furthermore, the cases where $F_\lambda$ or $G_\lambda$ is not unimodal still seem to be well-behaved:
\begin{conjecture}\label{nearlyunimodal}
If $F_\lambda$ or $G_\lambda$ is not unimodal, then there is some integer $n$ so that its coefficient sequence is increasing up to $q^{2n}$ and decreasing from $q^{2n+2}$.
\end{conjecture}
Note that this conjecture for $G_\lambda$ is equivalent to Observations 1-3 of \cite{Stanton1990}, while for $F_\lambda$ it contradicts a conjecture in \cite{StanleyZanello2013}.

For the $b=4$, $\lambda_4\approx\lambda_1$ case, an analysis similar to Theorem~\ref{5nearrectunimodal} shows that $[q^n](1-q)(1-q^2)F_\lambda(q)<0$ when $n\approx 2\lambda_1$ and that $[q^n](1-q)^2F_\lambda(q)<0$ when $n\approx 2\lambda_1$ and $n$ is odd. However, these two facts together are still insufficient to prove Conjecture~\ref{nearlyunimodal} even in this limited case.

We note that \eqref{Fgenfunc} can be used to get analogues of Proposition~\ref{equaldecreasingstrong} as well as a concavity result near $b=\frac{b}{2}\lambda_1$. Some sort of analogue of Proposition~\ref{equalincreasingweak} could then be used to prove the next conjecture.
\begin{conjecture}\label{bnearrectdistinctunimodal}
For every integer $b\ge 5$, there exists $\epsilon>0$ such that if $\lambda$ is a partition with $b$ distinct parts satisfying $\lambda_1\gg b$ and $\lambda_b>(1-\epsilon)\lambda_1$, then $F_\lambda$ is unimodal.
\end{conjecture}

Other classes of partitions that are likely to be approachable with the methods in this paper are partitions whose Ferrers diagram is the union of the Ferrers diagrams of $\lambda$ and the transpose of $\mu$, where $\lambda$ and $\mu$ are covered by Theorem~\ref{bnearrectunimodal}, and partitions whose parts are close to small multiples of some integer $a\gg b$.

Finally, there is a particularly interesting class of partitions for this problem: those whose parts form an arithmetic progression where $\lambda_b$ is at most the common difference (see Conjecture 2 of \cite{Stanton1990} and Conjectures 3.5 and 3.10 of \cite{StanleyZanello2013}).

\section{Acknowledgments}
This research was conducted as part of the University of Minnesota Duluth REU program, supported by NSA grant H98230-13-1-0273 and NSF grant 1358659. I would like to thank Joe Gallian for his advice and support. I would also like to thank participants of the REU program, particularly Alex Lombardi, for discussions about Lemma~$\ref{nonzeroaltsum}$.
\bibliographystyle{plain}
\bibliography{zbarskybib}
\end{document}